\documentclass[preprint,12pt]{elsarticle}
\usepackage{amsmath}
\usepackage{amssymb}
\usepackage{amsthm}
\usepackage{mathtools}
\usepackage{commath}
\usepackage{breqn}
\usepackage{bm}
\usepackage{graphicx}
\graphicspath{ {./figures/} }
\usepackage{xcolor}
\usepackage[rightcaption]{sidecap}
\usepackage{wrapfig}
\usepackage{float}
\usepackage{subcaption}
\usepackage{booktabs}
\usepackage{tabularx}
\usepackage{longtable}
\usepackage{tabu}
\usepackage{multirow}
\usepackage{makecell}
\usepackage[english]{babel}
\usepackage{imakeidx}
\usepackage{comment}
\usepackage[titletoc]{appendix}
\usepackage{resizegather}
\usepackage{xfrac}
\usepackage{verbatim}
\usepackage{lscape}
\usepackage{relsize}
\usepackage{enumitem}
\usepackage{textcomp}
\usepackage{doi}
\usepackage{fancyhdr}
\usepackage{algpseudocode}
\usepackage{setspace}
\usepackage{footnote}
\usepackage[makeroom]{cancel}
\usepackage[disable]{todonotes}
\usepackage[framemethod=tikz]{mdframed}
\usepackage{siunitx}
\usepackage{soul}
\usepackage[ruled,vlined]{algorithm2e}
\PassOptionsToPackage{hyphens}{url}
\usepackage{hyperref}
\hypersetup{colorlinks,linkcolor={blue},citecolor={blue},urlcolor={blue}}
\usepackage[numbers]{natbib}
\usepackage{cleveref}

\newcommand{\f}[2]{\frac{#1}{#2}}
\newcommand{\mb}[1]{\mathbf{#1}}

\DeclareMathAlphabet\mathbfcal{OMS}{cmsy}{b}{n}
\renewcommand{\d}{\mathop{}\!\mathrm{d}} 

\definecolor{rev1}{HTML}{FF999A} 
\definecolor{rev2}{HTML}{F3F298} 
\definecolor{rev3}{HTML}{B2E0AE} 
\definecolor{other}{HTML}{C8C7FF} 

\theoremstyle{plain}
\newtheorem{theorem}{Theorem}[section]      
\newtheorem{lemma}[theorem]{Lemma}

\theoremstyle{definition}

\theoremstyle{remark}
\newtheorem{remark}[theorem]{Remark}

\begin{document}

\begin{frontmatter}

\title{Torus Time-Spectral Method for \\
Quasi-Periodic Problems}

\author[utk]{Sicheng He\corref{cor1}}
\ead{sicheng@utk.edu}

\author[ripon]{Hang Li}
\ead{lihang@ripon.edu}

\author[utk]{Kivanc Ekici}
\ead{ekici@utk.edu}

\cortext[cor1]{Corresponding author.}
\address[utk]{Department of Mechanical and Aerospace Engineering, University of Tennessee, Knoxville, TN 37996}
\address[ripon]{Department of Physics and Engineering, Ripon College, Ripon, WI 54971}

\begin{abstract}
Quasi-periodic trajectories with two or more incommensurate frequencies are ubiquitous in nonlinear dynamics, yet the classical Fourier-based time-spectral method is tied to strictly periodic responses.
We introduce a torus time-spectral method that lifts the governing equations to an extended angular phase space, applies double-Fourier collocation on the invariant torus, and solves  for the state.
The formulation exhibits spectral convergence for quasi-periodic problem which we give a rigorous mathematical proof and also verify numerically. 
We demonstrate the approach on Duffing oscillators and a nonlinear Klein--Gordon system, documenting spectral error decay on the torus and tight agreement with time-accurate integrations while using modest frequency grids.
The method extends naturally to higher-dimensional tori and offers a computationally efficient framework for analyzing quasi-periodic phenomena in fluid mechanics, plasma physics, celestial mechanics, and other domains where multi-frequency dynamics arise.
\end{abstract}

\begin{keyword}
Time-spectral method \sep quasi-periodic systems \sep torus dynamics \sep Fourier collocation \sep nonlinear dynamics
\end{keyword}

\end{frontmatter}

\section{Introduction}

Quasi-periodic dynamics pervade many branches of computational physics, including celestial mechanics, plasma confinement, synchronization theory, and nonlinear structural vibrations \citep{Laskar1993,Jorba1997,Pikovsky2001,Meiss2007}.
Such trajectories combine two or more incommensurate frequencies and are naturally parameterized on a torus $\mathbb{T}^2$ via $(\theta_1,\theta_2) = (\omega_0 t,\omega_1 t)$; the viewpoint generalizes to higher-dimensional tori when additional modulation mechanisms are present.
Comparable multi-tone signatures arise in nonlinear oscillators and structural systems where cubic stiffness, fluid coupling, and modal interactions generate slow envelopes riding on high-frequency carriers \citep{Nayfeh1979,Nayfeh1995}.

In aeroelastic and aerodynamic analyses, quasi-periodicity emerges when a fast structural or acoustic oscillation is modulated by a slower flow or control-loop response.
Global-instability analyses, high-fidelity simulations, and targeted experiments on swept-wing buffet all reveal a dominant oscillation modulated by low-frequency motion, producing the characteristic sidebands and spectral shoulders associated with quasi-periodicity \citep{Crouch2009,Deck2005,Dandois2018,Timme2020}.
The breadth of these applications places a premium on numerical formulations that capture multi-tone responses without abandoning the accuracy and solver infrastructure developed for periodic orbits.

Time-spectral method (TSM) and harmonic balance (HB) techniques have become standard tools for periodic simulations in computational science.
By representing $u(t)$ with a truncated Fourier basis and modeling the governing equations at collocation points, the classical TSM recasts long integrations into steady nonlinear systems that benefit from mature steady-state solvers \citep{Hall2002,Gopinath2005}.
The approach inherits the spectral accuracy of the underlying basis functions, as established in the foundational texts on Fourier and Chebyshev methods \citep{Canuto2006,Boyd2001,Trefethen2000,Shen2011,Peyret2002,EKICI2020109560}.

Within the rotor--stator interaction (RSI) community, harmonic balance analyses have been deployed to capture turbomachinery limit-cycle oscillations and nonlinear aeroelastic phenomena \citep{Ekici_Hall2007,Ekici_Hall2008,Ekici2011}.
Building on that foundation, one-shot frameworks have been extended to multi-degree-of-freedom and viscous transonic problems, enabling rapid predictions of aeroelastic dynamics for canonical airfoils and wings \citep{Li2017,Li2018d,Li2019e}.
Yet the classical TSM enforces strict periodicity, making it ill-suited for trajectories with multiple incommensurate tones.
Closing this gap requires formulations that retain the algorithmic advantages of TSM while solving directly for multi-frequency responses and their invariant manifolds.

Aeroelastic demonstrations show that the approach captures limit cycles and forced responses while remaining compatible with large-scale solvers \citep{Thomas2002}.
Complementary advances have produced coupled Newton--Krylov time-spectral solvers and adjoint formulations that target flutter prediction and time-spectral aerodynamic optimization \citep{HUANG2014481,He2018a,He2019b,He2021c}.
Recent adjoint-based strategies further extend these tools to quantify limit-cycle oscillation stability sensitivities and support active suppression mechanisms \citep{He2023a}.

Several extensions of TSM and HB methods have been developed for quasi-periodic responses.
Incremental harmonic balance algorithms in circuit theory first populated two-dimensional frequency lattices to compute almost-periodic steady states \citep{Chua1981}.
Subsequent refinements introduced multiple time scales to treat aperiodic vibrations in structural mechanics \citep{Lau1983}.
The electronics community later adapted these ideas to handle near-commensurate tones in microwave circuits \citep{Kundert1988}.
Mechanical and rotor-dynamics studies extended multi-tone balance to assess response and stability of nonlinear shafts and oscillators \citep{Kim1996,Liu2007a}.
Variable-coefficient formulations improved robustness for structural vibrations driven by several interacting tones \citep{Zhou2015}.
Continuation-based harmonic balance implementations track invariant tori across parameter sweeps and quantify their stability \citep{Guillot2017}.
Comparative studies benchmark multi- and variable-coefficient strategies on representative multi-frequency test cases \citep{Wu2023}.
Their implementation---like the broader HB family---operates entirely in the frequency domain, requiring the governing equations to be rewritten for Fourier coefficients and the linearized operators that couple them; this frequency-domain reformulation complicates direct reuse of existing time-marching residual evaluations and introduces additional bookkeeping whenever the set of active tones changes.
Recent work emphasizes efficient automatic tone selection when computing high-order quasi-periodic responses \citep{Wang2023}.
Hysteretic systems motivate additional enhancements to handle strong nonlinearities within the harmonic balance setting \citep{Balaji2024}.
Within computational fluid dynamics (CFD), supplemental-frequency harmonic balance (SF-HB) augments the base tone set to approximate broadband aerodynamic response \citep{Li2021}.
Resolvent-based analyses identify compact frequency sets that govern coherent structures in jets and similar flows \citep{Padovan2022}.

Unlike previous work reported in the literature, this study focuses on a novel \emph{Torus Time-Spectral Method (TTSM)}, which transforms the multi-frequency problem to the torus $\mathbb{T}^2$ and solves the invariance equation directly with double-Fourier collocation.
A key advantage of the proposed formulation is that it operates entirely in the \emph{time domain}: the governing equations are collocated at discrete points on the torus without requiring an explicit transformation to Fourier coefficients.
This time-domain perspective allows existing residual routines---originally developed for time-accurate or single-frequency time-spectral solvers---to be reused with minimal modification; only the spectral differentiation operator changes when moving from periodic to quasi-periodic problems.
In contrast, traditional harmonic balance implementations require a frequency-domain reformulation that couples Fourier modes through convolution sums or alternating frequency--time (AFT) procedures, complicating code maintenance and limiting modularity.

The specific contributions of this paper are twofold: first, a time-domain torus collocation formulation that enables direct reuse of existing residual routines from steady-state or single-frequency time-spectral solvers; and second, a rigorous convergence analysis together with a verification campaign, summarized in \Cref{sec:convergence,sec:results_figures}, that establishes spectral convergence on quasi-periodic ODEs and PDEs, contrasts the torus-based accuracy with the algebraic behavior of rational-period surrogates \citep{Li2021}, and highlights the significant reduction in degrees of freedom compared with supplemental-frequency harmonic balance for near-commensurate forcing.
In this paper, we focus on non-autonomous quasi-periodic systems with known frequencies.
For autonomous system, especially ones arising from the torus-attractor, the frequencies are in general unknown.
However, it is likely that such problem can be formulated with additional equations following Thomas et al.'s work that extends TSM to TSM LCO equation \cite{Thomas2002}.

The remainder of the paper is organized as follows.
\Cref{sec:periodic_vs_quasiperiodic} revisits periodic and quasi-periodic dynamics, emphasizing the lifting procedure to the torus.
\Cref{sec:classic_tsm} summarizes the classical TSM formulation, while \Cref{sec:ttsm} introduces the proposed TTSM framework. 
\Cref{sec:convergence} assesses the convergence of TTSM on representative problems, and \Cref{sec:results_figures} discusses the numerical results and the refined figure set.
Conclusions are drawn in \Cref{sec:conclusion}.

\section{Periodic and quasi-periodic dynamics}
\label{sec:periodic_vs_quasiperiodic}

The time-spectral method was developed with significant success for periodic problems in the CFD community, often yielding an order-of-magnitude speed-up compared with time-accurate integration.
Applications range from forced periodic responses to limit-cycle oscillations (LCO) \citep{Hall2002,Gopinath2005,Thomas2002}.

However, the original form of TSM cannot be used to capture quasi-periodic regimes whose Fourier spectra contain incommensurate peaks.
Simple rational approximations \citep{Li2021} introduce ``beat'' phenomena when the base frequencies are nearly commensurate, motivating the torus lifting adopted in this work.
By recognizing that quasi-periodic trajectories densely fill an $n$-torus without repetition, we reformulate the problem within a higher-dimensional periodic space.
We now revisit the lifting strategy in a compact notation that distinguishes temporal and angular variables throughout the remainder of the paper.

Consider a nonlinear dynamical system governed by
\begin{equation}
\label{eq:dynamics}
\dot{\mb{q}}(t) = \mb{f}\!\big(\mb{q}(t),t\big),
\end{equation}
where $\mb{q}\in\mathbb{R}^n$ denotes the state vector, $\mb{f}: \mathbb{R}^{n}\times\mathbb{R}\rightarrow \mathbb{R}^n$ is the forcing term, and $n$ is the dimension of the state space.
The autonomous system $\dot{\mb{q}}=\mb{f}(\mb{q})$ is a special case.

\subsection{Periodic solutions}
A solution $\mb{q}(t)$ is called \emph{periodic} with period $T>0$ if $\mb{q}(t+T) = \mb{q}(t)$ for all $t\in\mathbb{R}$.
Equivalently, the solution admits a representation
\begin{equation}
\mb{q}(t) = \mb{q}(\theta), \qquad \theta = \omega t, \quad \omega = \f{2\pi}{T},
\label{eq:periodic_def}
\end{equation}
where $\mb{q}(\theta)$ is $2\pi$-periodic in the angular variable $\theta$.
Periodic trajectories lie on a one-dimensional invariant manifold (a circle $\mathbb{T}^1$).
Differentiating \cref{eq:periodic_def} with respect to time yields the \emph{invariance equation} on $\mathbb{T}^1$:
\begin{equation}
\omega \,\frac{\d \mb{q}}{\d \theta}(\theta)
= \mb{f}\!\big(\mb{q}(\theta), \theta\big),
\label{eq:periodic_invariance}
\end{equation}
which must hold for all $\theta\in[0,2\pi)$.
Solving \cref{eq:periodic_invariance} for $\mb{q}(\theta)$ and $\omega$ simultaneously recovers the periodic orbit without time-marching, and this is the basis of the classical time-spectral method.

\subsection{Quasi-periodic solutions}
A solution $\mb{q}(t)$ is called \emph{quasi-periodic} with $m$ basic frequencies $\omega_1,\ldots,\omega_m$ if it admits a representation
\begin{equation}
\mb{q}(t) = \mb{q}(\theta_1,\ldots,\theta_m), \qquad \theta_j = \omega_j t,\quad j=1,\ldots,m,
\label{eq:quasiperiodic_def}
\end{equation}
where $\mb{q}(\boldsymbol{\theta})$ is a smooth $2\pi$-periodic function on the $m$-dimensional torus $\mathbb{T}^m=[0,2\pi)^m$, and the frequency vector $\boldsymbol{\omega}=(\omega_1,\ldots,\omega_m)^\intercal$ is \emph{rationally independent}, meaning that
\begin{equation}
\sum_{j=1}^{m} k_j \omega_j = 0
\quad\text{for integers}\quad k_j\in\mathbb{Z}
\quad\Longrightarrow\quad
k_1=\cdots=k_m=0.
\label{eq:incommensurate}
\end{equation}
When the frequencies are rationally independent, the trajectory $\{\mb{q}(t):t\in\mathbb{R}\}$ densely fills the invariant torus $\{\mb{q}(\boldsymbol{\theta}):\boldsymbol{\theta}\in\mathbb{T}^m\}$ and never repeats.
Periodic solutions are the special case of $m=1$. 

For the remainder of this work we focus on the two-frequency case ($m=2$), although the formulation extends naturally to arbitrary $m$.
Differentiating \cref{eq:quasiperiodic_def} with respect to time and applying the chain rule yields the \emph{invariance equation} on the two-torus $\mathbb{T}^2$:
\begin{equation}
\omega_1 \,\partial_{\theta_1} \mb{q}(\theta_1,\theta_2)
+ \omega_2 \,\partial_{\theta_2} \mb{q}(\theta_1,\theta_2)
= \mb{f}\!\big(\mb{q}(\theta_1,\theta_2), \theta_1,\theta_2\big),
\label{eq:torus_invariance}
\end{equation}
which must hold for all $(\theta_1,\theta_2)\in\mathbb{T}^2$.
This equation states that the vector field $\mb{f}$ evaluated on the torus must be tangent to the torus, ensuring that trajectories remain on the invariant manifold for all time.
Solving \cref{eq:torus_invariance} for $\mb{q}(\theta_1,\theta_2)$ and $(\omega_1,\omega_2)$ simultaneously recovers the quasi-periodic solution without integrating the original time-dependent system \cref{eq:dynamics}.

\paragraph{Example}
Consider the scalar differential equation
\begin{equation}
\dot{q}(t) = \sin(\omega_1 t) + \cos(\omega_2 t),
\label{eq:example_ode}
\end{equation}
with rationally independent frequencies $\omega_1,\omega_2>0$.
Writing $q(t)=q(\theta_1,\theta_2)$ with $\theta_j=\omega_j t$, the invariance equation, i.e., \cref{eq:torus_invariance}, becomes
\begin{equation}
\omega_1 \,\partial_{\theta_1} q(\theta_1,\theta_2)
+ \omega_2 \,\partial_{\theta_2} q(\theta_1,\theta_2)
= \sin(\theta_1) + \cos(\theta_2),
\label{eq:example_torus}
\end{equation}
which is a first-order partial differential equation (PDE) on $\mathbb{T}^2$ rather than an ordinary differential equation (ODE) in time.
This \emph{lifting} from the temporal variable $t$ to the angular coordinates $(\theta_1,\theta_2)$ is the foundation of the torus time-spectral method.

\section{Classic time-spectral method}
\label{sec:classic_tsm}

The classical time-spectral method leverages the spectral convergence of the Fourier transform to represent a periodic signal. The key feature of this method is the conversion of a periodic time signal to the frequency space for differentiation before casting it back to the time domain.  These operations can be efficiently performed in one step using the spectral differentiation operator $\mb{D}$ defined as follows:
\begin{equation}
\mb{D}[j, k]=
\begin{cases}
0, & j = k, \\
\dfrac{(-1)^{\,j-k}}{2}\,
\csc\!\left( \dfrac{(j-k)\pi}{n_{\text{TS}}} \right), & j \ne k,
\end{cases}.
\label{eq:time-spectral-D}
\end{equation}
For a periodic problem with a known fundamental frequency $\omega$, the classical TSM formulation discretizes the periodic orbit at $n_{\text{TS}}$ equally spaced collocation points within one period $T=2\pi/\omega$.
Let $\theta_{(i)} = 2\pi i / n_{\text{TS}}$ for $i=0,\ldots,n_{\text{TS}}-1$ denote the phase instances, and let $\mb{q}_{(i)} \in \mathbb{R}^n$ be the state vector at phase $\theta_{(i)}$.
The discrete residual at each collocation point balances the spectral phase derivative with the forcing term:
\begin{equation}
\mb{r}^{(n_{\text{TS}})}_{\mathrm{TS}}(\mb{q}^{(n_{\text{TS}})}) =
\,\omega\left(\mb{D} \otimes \mb{I}_{n}\right)\,\mb{q}^{(n_{\text{TS}})} - \mb{f}^{(n_{\text{TS}})}\left(\mb{q}^{(n_{\text{TS}})}, \boldsymbol{\theta}^{(n_{\text{TS}})}\right),
\end{equation}
where the stacked vectors are defined as
\begin{equation}
\boldsymbol{\theta}^{(n_{\text{TS}})} = \begin{bmatrix}\theta_{(0)}\\ \theta_{(1)}\\ \vdots\\ \theta_{(n_{\text{TS}}-1)}\end{bmatrix}, \quad \mb{q}^{(n_{\text{TS}})} = \begin{bmatrix}\mb{q}_{(0)}\\ \mb{q}_{(1)}\\ \vdots\\ \mb{q}_{(n_{\text{TS}}-1)}\end{bmatrix}, \quad \mb{f}^{(n_{\text{TS}})} = \begin{bmatrix}\mb{f}_{(0)}\\ \mb{f}_{(1)}\\ \vdots\\ \mb{f}_{(n_{\text{TS}}-1)}\end{bmatrix}, \quad \mb{r}^{(n_{\text{TS}})} = \begin{bmatrix}\mb{r}_{(0)}\\ \mb{r}_{(1)}\\ \vdots\\ \mb{r}_{(n_{\text{TS}}-1)}\end{bmatrix}.
\end{equation}
Here $\mb{f}_{(i)} = \mb{f}(\mb{q}_{(i)}, \theta_{(i)})$ evaluates the forcing at the $i$-th collocation point, and the tensor product $\mb{D} \otimes \mb{I}_n$ applies the spectral differentiation matrix to each component of the state vector.
Solving $\mb{r}^{(n_{\text{TS}})}_{\mathrm{TS}}=0$ yields the periodic solution at all collocation nodes simultaneously, bypassing the need for long-time integration.

When approximating a quasi-periodic problem using the classical time-spectral method, spectral convergence can be achieved locally.
However, this is not the case over a dense torus.
This limitation motivates the development of the TTSM method.

\section{Torus Time-Spectral Method (TTSM)}
\label{sec:ttsm}

\subsection{Nodal equation}
Using the torus coordinates $\theta_1 = \omega_1 t, \; \theta_2 = \omega_2 t$,
the invariance equation \eqref{eq:torus_invariance} on $\mathbb{T}^2$ becomes our starting point.
We discretize each angle over $n_{\text{TS},1}$ and $n_{\text{TS},2}$ equally-spaced points respectively:
\begin{equation}
\theta_1^{(\ell)} = \frac{2\pi \ell}{n_{\text{TS},1}}, \quad \ell=0,\dots,n_{\text{TS},1}-1,
\qquad
\theta_2^{(m)} = \frac{2\pi m}{n_{\text{TS},2}}, \quad m=0,\dots,n_{\text{TS},2}-1.
\end{equation}
At each discrete point $(\ell,m)$ in the angular space, the nodal unknowns are denoted by
$\mb{q}_{(\ell,m)} \in \mathbb{R}^n$ whereas
the Fourier spectral differentiation matrices in $\theta_1$ and $\theta_2$ are denoted by
$\mb{D}_1 \in \mathbb{R}^{n_{\text{TS},1}\times n_{\text{TS},1}}$ and
$\mb{D}_2 \in \mathbb{R}^{n_{\text{TS},2}\times n_{\text{TS},2}}$, each constructed via \cref{eq:time-spectral-D} with the corresponding grid size.
The equation at each node then becomes
\begin{equation}
\mb{r}_{\ell,m} =
\omega_1 \sum_{p=0}^{n_{\text{TS},1}-1} (\mb{D}_1)_{\ell p}\, \mb{q}_{(p,m)}
+ \omega_2 \sum_{q=0}^{n_{\text{TS},2}-1} (\mb{D}_2)_{m q}\, \mb{q}_{(\ell,q)}
- \mb{f}\!\big(\mb{q}_{(\ell,m)}, \theta_1^{(\ell)}, \theta_2^{(m)}\big)
= 0,
\label{eq:ts-collocation}
\end{equation}
which couples each nodal state to its neighbors through the spectral stencil depicted by the example of \cref{fig:ttsm_stencil}.

\subsection{Stacked equation}
For global Newton--Krylov solvers, we work with a stacked representation.
Let $N = n \times n_{\text{TS},1} \times n_{\text{TS},2}$ denote the total number of unknowns.
To construct the global differentiation operators, we extend the spectral matrices $\mb{D}_1$ and $\mb{D}_2$ to the full tensor-product grid and then to the state dimension via
\begin{equation}
{\mathcal{D}}_1 = (\mb{I}_{n_{\text{TS},2}} \otimes \mb{D}_1) \otimes \mb{I}_n \in \mathbb{R}^{N \times N},
\qquad
{\mathcal{D}}_2 = (\mb{D}_2 \otimes \mb{I}_{n_{\text{TS},1}}) \otimes \mb{I}_n \in \mathbb{R}^{N \times N}.
\label{eq:stacked-operators}
\end{equation}
The discrete residual can then be written as
\begin{equation}
\mb{r}^{(n_{\text{TS},1},n_{\text{TS},2})}_{\mathrm{TTSM}}(\mb{q}^{(n_{\text{TS},1},n_{\text{TS},2})})
= \Big( \omega_1 {\mathcal{D}}_1 + \omega_2 {\mathcal{D}}_2 \Big) \mb{q}^{(n_{\text{TS},1},n_{\text{TS},2})}
- \mb{f}^{(n_{\text{TS},1},n_{\text{TS},2})}(\mb{q}^{(n_{\text{TS},1},n_{\text{TS},2})}, \boldsymbol{\theta}^{(n_{\text{TS},1},n_{\text{TS},2})}),
\label{eq:ts-residual}
\end{equation}
where the stacked vectors are defined as
\begin{equation}
\begin{aligned}
\boldsymbol{\theta}^{(n_{\text{TS},1},n_{\text{TS},2})} &= \begin{bmatrix}\boldsymbol{\theta}_{(0,0)}\\ \boldsymbol{\theta}_{(1,0)}\\ \vdots\\ \boldsymbol{\theta}_{(n_{\text{TS},1}-1,n_{\text{TS},2}-1)}\end{bmatrix}, \quad
\mb{q}^{(n_{\text{TS},1},n_{\text{TS},2})} = \begin{bmatrix}\mb{q}_{(0,0)}\\ \mb{q}_{(1,0)}\\ \vdots\\ \mb{q}_{(n_{\text{TS},1}-1,n_{\text{TS},2}-1)}\end{bmatrix}, \\
\mb{f}^{(n_{\text{TS},1},n_{\text{TS},2})} &= \begin{bmatrix}\mb{f}_{(0,0)}\\ \mb{f}_{(1,0)}\\ \vdots\\ \mb{f}_{(n_{\text{TS},1}-1,n_{\text{TS},2}-1)}\end{bmatrix}, \quad
\mb{r}^{(n_{\text{TS},1},n_{\text{TS},2})} = \begin{bmatrix}\mb{r}_{(0,0)}\\ \mb{r}_{(1,0)}\\ \vdots\\ \mb{r}_{(n_{\text{TS},1}-1,n_{\text{TS},2}-1)}\end{bmatrix}.
\end{aligned}
\end{equation}
Here $\boldsymbol{\theta}_{(\ell,m)} = (\theta_1^{(\ell)}, \theta_2^{(m)})$ denotes the angular grid point, $\mb{f}_{(\ell,m)} = \mb{f}(\mb{q}_{(\ell,m)}, \theta_1^{(\ell)}, \theta_2^{(m)})$ is the forcing at the $(\ell,m)$-th collocation point, and the tensor products ${\mathcal{D}}_j$ apply the spectral differentiation matrices to each component of the state vector.
Solving for $\mb{r}^{(n_{\text{TS},1},n_{\text{TS},2})}_{\mathrm{TTSM}}=0$ yields the quasi-periodic solution at all collocation nodes simultaneously.

For simplicity, we define the compact form by
\begin{equation}
\mb{r}_{\mathrm{TTSM}}(\hat{\mb{q}})
= \Big( \omega_1 {\mathcal{D}}_1 + \omega_2 {\mathcal{D}}_2 \Big) \hat{\mb{q}}
- \mb{f}(\hat{\mb{q}}, \boldsymbol{\theta}),
\label{eq:ts-residual-compact}
\end{equation}
where $\hat{\mb{q}} \in \mathbb{R}^{N}$ denotes the stacked discrete nodal vector.

\begin{figure}[htbp]
\centering
\includegraphics[width=0.5\linewidth]{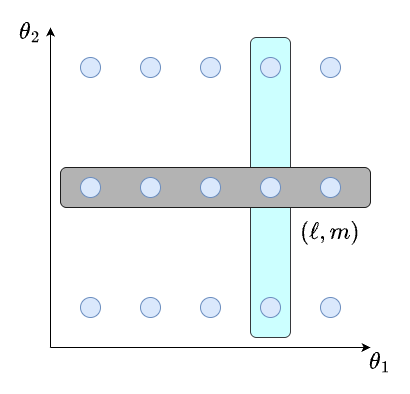}
\caption{Collocation stencil for the TTSM discretization on the two-torus $\mathbb{T}^2$ (example: $n_{\text{TS},1} = 5$ and $n_{\text{TS},2} = 3$). Each circle represents a collocation node $(\theta_1^{(\ell)}, \theta_2^{(m)})$. The residual at node $(\ell, m)$ (highlighted) depends on the solution values along the entire row (green, constant $\theta_2$) and column (yellow, constant $\theta_1$) due to the global coupling of the Fourier spectral differentiation operators $\mathcal{D}_1$ and $\mathcal{D}_2$.}
\label{fig:ttsm_stencil}
\end{figure}

\subsection{Rank deficiency}
When the governing dynamics contain a neutral direction---for example, the forced linear oscillator with no restoring term shown in \cref{fig:forced_dual_frequency}---the collocation system inherits the resulting gauge freedom: any constant shift of $\hat{\mb{q}}$ leaves \cref{eq:ts-collocation} unchanged and the discrete Jacobian loses rank.
In such cases we regularize by enforcing a \emph{nodal anchor}, replacing one row of \cref{eq:ts-collocation} with a physical constraint such as $\mb{q}_{(0,0)}=\mb{q}(0)$ or by prescribing the torus average; this replacement removes the nullspace that would otherwise render the linearized system singular.
For the nonlinear and damped examples considered later (Duffing oscillator, Klein--Gordon equation), the combination of nonlinear saturation and dissipation determines a unique amplitude, so no additional constraint is needed.

\subsection{Extension to $k$-torus systems}
The proposed framework extends naturally to quasi-periodic problems with $k$ frequency components.
Let $\boldsymbol{\theta} = (\theta_1,\ldots,\theta_k)$ with $\theta_j = \omega_j t$ parameterize the $k$-torus $\mathbb{T}^k$, and discretize each angle with $n_{\text{TS},j}$ collocation points.
The resulting residual maintains the Kronecker-sum structure
\begin{equation}
\mb{r}(\hat{\mb{q}}) = \sum_{j=1}^{k} \omega_j \mathcal{D}_j \hat{\mb{q}}
- \mb{f}(\hat{\mb{q}}, \boldsymbol{\theta}),
\end{equation}
where $\mathcal{D}_j = (\cdots \otimes \mb{I}_{M_{j+1}} \otimes \mb{D}_j \otimes \mb{I}_{M_{j-1}} \otimes \cdots) \otimes \mb{I}_n$ is the lifted differentiation matrix in direction $\theta_j$.
All ingredients of the two-frequency solver carry over: the Jacobian remains block-circulant with respect to each angular index, the preconditioner is assembled from one-dimensional TSM solves, and the convergence results in \cref{sec:convergence} generalize by replacing the bivariate interpolant with its multivariate counterpart.
Finally, padding can be applied to suppress aliasing for problems with more than two incommensurate frequencies.

\subsection{Solution method}
We solve the nonlinear residual given in \cref{eq:ts-residual} using a damped Newton--Krylov method with the base frequencies $(\omega_1,\omega_2)$ held fixed.
At iteration $k$ the torus state $\hat{\mb{q}}^{(k)}$ is updated by solving
\begin{equation}
\mb{J}^{(k)} \,\Delta\hat{\mb{q}}
= -\,\mb{r}(\hat{\mb{q}}^{(k)}),
\label{eq:nk_linear_system}
\end{equation}
where $\mb{J}^{(k)}$ is the Jacobian of the residual.
The Jacobian exhibits a block-sparse structure arising from the global coupling of the Fourier spectral stencil, as illustrated in \cref{fig:ttsm_stencil}.
As mentioned earlier, if a nodal anchor or linear phase constraint is enforced, the corresponding row replaces one equation so that \cref{eq:nk_linear_system} remains well posed.
Restarted GMRES technique \citep{Saad2003} is used to solve the linear system, and a backtracking line search on the Newton update ensures a monotonic reduction of the nonlinear residual.

\begin{figure}[htbp]
\centering
\includegraphics[width=0.5\linewidth]{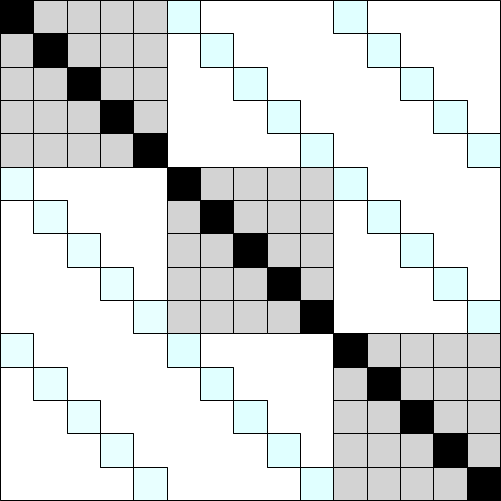}
\caption{Sparsity pattern of the Jacobian matrix $\mb{J}$ for the TTSM discretization corresponding to the $5 \times 3$ stencil in \cref{fig:ttsm_stencil}. 
The block-circulant structure arises from the global row and column coupling of the Fourier spectral operators and enables efficient preconditioning via FFT-based techniques.
Black blocks represent the spatial residual contribution, grey blocks the $\theta_1$-direction coupling, and turquoise blocks the $\theta_2$-direction coupling.}
\label{fig:block_sparse}
\end{figure}

\section{Convergence analysis}
\label{sec:convergence}

Next, we analyze the discrete residual of the TTSM.
Recall that $n$ denotes the state-space dimension and let $N = n \times n_{\text{TS},1} \times n_{\text{TS},2}$ be the total number of unknowns.
Let $\mb{q}^\star: \mathbb{T}^2 \to \mathbb{R}^{n}$ denote the exact (continuous) torus solution of \cref{eq:torus_invariance}.
We discretize on an $n_{\text{TS},1} \times n_{\text{TS},2}$ collocation grid with nodes $(\theta_1^{(\ell)}, \theta_2^{(m)})$ for $\ell = 0, \ldots, n_{\text{TS},1}-1$ and $m = 0, \ldots, n_{\text{TS},2}-1$.
The discrete state vector $\hat{\mb{q}} \in \mathbb{R}^{N}$ collects the nodal values, and the discrete residual is
\begin{equation}
\mb{r}(\hat{\mb{q}})
= (\omega_1 {\mathcal{D}}_1+\omega_2 {\mathcal{D}}_2)\,\hat{\mb{q}}
-\mb{f}(\hat{\mb{q}}, \boldsymbol{\theta}),
\end{equation}
where $\mathcal{D}_j \in \mathbb{R}^{N \times N}$ are the Fourier spectral differentiation matrices acting on the nodal values.

Recall that $\mathbb{T}^2 = [0,2\pi)^2$ denotes the two-torus.
Let $\mathbb{T}_{n_{\text{TS},1},n_{\text{TS},2}}$ denote the space of trigonometric polynomials on $\mathbb{T}^2$ with bounded wave numbers,
\begin{equation}
\mathbb{T}_{n_{\text{TS},1},n_{\text{TS},2}}
= \operatorname{span}\!\left\{e^{\mathrm{i}(k_1\theta_1+k_2\theta_2)} : |k_1|\le K_1, |k_2|\le K_2 \right\},
\qquad K_j = \lfloor (n_{\text{TS},j}-1)/2 \rfloor,
\end{equation}
Denote by $\mb{P} : C(\mathbb{T}^2; \mathbb{R}^{n}) \to \mathbb{R}^{N}$ the sampling operator that evaluates a continuous function at the collocation nodes and stacks the values into a vector.
Let $\mb{v} : \mathbb{T}^2 \to \mathbb{R}^{n}$ be the unique vector-valued trigonometric polynomial in $(\mathbb{T}_{n_{\text{TS},1},n_{\text{TS},2}})^n$ whose nodal values coincide with the samples of $\mb{q}^\star$ at the collocation grid, and write $\hat{\mb{v}} \coloneq \mb{P}\mb{q}^\star \in \mathbb{R}^{N}$ for these nodal values.

\begin{lemma}[Truncation error and consistency]
\label{lem:consistency}
Assume $\mb{q}^\star \in C^s(\mathbb{T}^2)$ with $s>1$, and that $\mb{f}(\cdot,\theta_1,\theta_2)$ is $C^1$ in $\mb{q}$ uniformly in $(\theta_1,\theta_2)$.
Then the discrete operator
\(
(\omega_1 {\mathcal{D}}_1+\omega_2 {\mathcal{D}}_2)
\)
is consistent with the continuous operator
\(
(\omega_1 \partial_{\theta_1}+\omega_2 \partial_{\theta_2})
\)
on $(\mathbb{T}_{n_{\text{TS},1},n_{\text{TS},2}})^n$, and the nodal truncation error
\begin{equation}
\mb{t}_{n_{\text{TS},1},n_{\text{TS},2}}
\coloneq (\omega_1 {\mathcal{D}}_1+\omega_2 {\mathcal{D}}_2)\,\hat{\mb{v}}
- \mb{f}(\hat{\mb{v}}, \boldsymbol{\theta}) \in \mathbb{R}^{N}
\end{equation}
satisfies
\begin{equation}
\|\mb{t}_{n_{\text{TS},1},n_{\text{TS},2}}\|_\infty
\le
C\big(n_{\text{TS},1}^{1-s}+n_{\text{TS},2}^{1-s}\big),
\qquad s>1,
\label{eq:algebraic-rate}
\end{equation}
where $C$ depends on $\|\mb{q}^\star\|_{C^s}$ and the Lipschitz constant of $\mb{f}$.
If $\mb{q}^\star$ and $\mb{f}$ are analytic in $(\theta_1,\theta_2)$ on a complex strip of half-width $\rho>0$, then
\begin{equation}
\|\mb{t}_{n_{\text{TS},1},n_{\text{TS},2}}\|_\infty
\le
C\,e^{-\rho\,\min\{n_{\text{TS},1},n_{\text{TS},2}\}},
\label{eq:spectral-rate}
\end{equation}
where $C$ depends on the analytic norm of $\mb{q}^\star$ on the strip.
\end{lemma}

\begin{proof}
The interpolant $\mb{v}$ of $\mb{q}^\star$ belongs to the trigonometric polynomial space $\mathbb{T}_{n_{\text{TS},1},n_{\text{TS},2}}$, so it admits the expansion
\begin{equation}
\mb{v}(\theta_1,\theta_2)=\sum_{|k_1|\le K_1}\sum_{|k_2|\le K_2} \hat{\mb{v}}_{k_1,k_2} e^{\mathrm{i}(k_1\theta_1+k_2\theta_2)}. 
\end{equation}
By construction, the differentiation matrices ${\mathcal{D}}_j$ reproduce the action of $\partial_{\theta_j}$ exactly on each Fourier mode in this span; therefore the discrete and continuous derivatives coincide at the collocation nodes,
\(
(\omega_1 {\mathcal{D}}_1+\omega_2 {\mathcal{D}}_2)\hat{\mb{v}}
= \bigl[(\omega_1 \partial_{\theta_1}+\omega_2 \partial_{\theta_2})\mb{v}\bigr]\big|_{\text{nodes}}.
\)
All unresolved Fourier content of $\mb{q}^\star$, therefore, enters only through the interpolation defect $\mb{v}-\mb{q}^\star$.
Evaluating the residual of the interpolant and subtracting the continuous invariance equation for $\mb{q}^\star$ yields, pointwise at the nodes,
\begin{multline}
\mb{t}_{n_{\text{TS},1},n_{\text{TS},2}}
= \bigl[(\omega_1 \partial_{\theta_1}+\omega_2 \partial_{\theta_2})(\mb{v}-\mb{q}^\star)\bigr]\big|_{\text{nodes}} \\
-\bigl[\mb{f}(\mb{v}, \boldsymbol{\theta})-\mb{f}(\mb{q}^\star, \boldsymbol{\theta})\bigr]\big|_{\text{nodes}}.
\end{multline}
Applying standard Fourier interpolation estimates on $\mathbb{T}^2$ yields
\begin{align}
\|\mb{v}-\mb{q}^\star\|_\infty
&\le C(n_{\text{TS},1}^{-s}+n_{\text{TS},2}^{-s}), \\
\|\partial_{\theta_j}(\mb{v}-\mb{q}^\star)\|_\infty
&\le C(n_{\text{TS},1}^{1-s}+n_{\text{TS},2}^{1-s}),
\end{align}
and by the mean-value theorem for $\mb{f}$,
\(
\|\mb{f}(\mb{v})-\mb{f}(\mb{q}^\star)\|
\le L_f \|\mb{v}-\mb{q}^\star\|.
\)
Combining these bounds yields \eqref{eq:algebraic-rate}; for analytic data, exponential decay of Fourier coefficients implies \eqref{eq:spectral-rate}.
\end{proof}

\begin{theorem}[Convergence of the TTSM solution for non-autonomous systems]
\label{thm:convergence}
Consider a non-autonomous system where the frequencies $(\omega_1,\omega_2)$ are prescribed.
Assume the hypotheses of Lemma~\ref{lem:consistency} and that the linearized operator $\mathcal{J}^\star \in \mathbb{R}^{N \times N}$,
\begin{equation}
\mathcal{J}^\star \coloneq
\omega_1 {\mathcal{D}}_1+\omega_2 {\mathcal{D}}_2
- \operatorname{blkdiag}\{\partial_{\mb{q}}\mb{f}(\mb{q}^\star(\theta_1,\theta_2),\theta_1,\theta_2)\},
\end{equation}
where each $\mathcal{D}_j \in \mathbb{R}^{N \times N}$ and the block-diagonal matrix has $n_{\text{TS},1} \times n_{\text{TS},2}$ blocks of size $n \times n$, satisfies a uniform inf–sup (stability) condition:
there exists $c_0>0$ such that for all grid sizes sufficiently large,
\begin{equation}
\|\mathcal{J}^\star \mb{z}\|_\infty \ge c_0 \|\mb{z}\|_\infty
\quad \text{for all } \mb{z} \in \mathbb{R}^{N}.
\end{equation}
Then for sufficiently large $(n_{\text{TS},1},n_{\text{TS},2})$ there exists a unique discrete solution
$\hat{\mb{q}}$ of $\mb{r}(\hat{\mb{q}})=0$ such that
\begin{equation}
\|\hat{\mb{q}}-\hat{\mb{v}}\|_\infty
\le
C\big(n_{\text{TS},1}^{1-s}+n_{\text{TS},2}^{1-s}\big),
\label{eq:convergence-algebraic}
\end{equation}
for $C^s$ data, or
\begin{equation}
\|\hat{\mb{q}}-\hat{\mb{v}}\|_\infty
\le
C\,e^{-\rho\,\min\{n_{\text{TS},1},n_{\text{TS},2}\}},
\label{eq:convergence-spectral}
\end{equation}
for analytic data.
\end{theorem}

\begin{proof}
Evaluating $\mb{r}(\hat{\mb{v}})$ gives
the truncation error $\mb{t}_{n_{\text{TS},1},n_{\text{TS},2}}$ from Lemma~\ref{lem:consistency}.
Since $\hat{\mb{v}} = \mb{P}\mb{q}^\star$, the Jacobian at $\hat{\mb{v}}$ coincides with $\mathcal{J}^\star$, which is invertible with $\|\mathcal{J}^{\star-1}\|_\infty \le 1/c_0$ by the stability assumption.
A standard linearization argument then yields
\[
\|\hat{\mb{q}} - \hat{\mb{v}}\|_\infty
\le \|\mathcal{J}^{\star-1}\|_\infty \|\mb{r}(\hat{\mb{v}})\|_\infty
\le \frac{1}{c_0}\|\mb{t}_{n_{\text{TS},1},n_{\text{TS},2}}\|_\infty,
\]
which gives \eqref{eq:convergence-algebraic}–\eqref{eq:convergence-spectral}.
\end{proof}

\begin{remark}[Stability assumption]
\label{rem:stability}
The inf--sup condition on $\mathcal{J}^\star$ is a discrete stability hypothesis that parallels the continuous non-resonance condition for the linearized flow on $\mathbb{T}^2$.
\end{remark}

\section{Numerical results}
\label{sec:results_figures}

\subsection{Forced linear oscillator}

We first study the forced oscillation with two incommensurate frequencies.
The response is simply
\begin{equation}
\dot{q} = \sin(\omega_0 t) + \cos(\omega_f t), \quad q(0) = 0.
\label{eq:examp_1}
\end{equation}
The analytic solution of this simple initial value problem is given by
\begin{equation}
q(t) = \f{1}{\omega_0} -\f{1}{\omega_0}\cos{(\omega_0 t)} + \f{1}{\omega_f}\sin{(\omega_f t)}.
\end{equation}
We solve the problem using the TTSM and the results are shown in \cref{fig:forced_dual_frequency}.
It turns out the TTSM formulation is singular for such linear forced cases meaning that one can only determine the solution $q$ up to a constant shift $c$.
To ensure we have a unique solution, the first row of the governing equation can be replaced by the initial condition, $\mb{q}_{(0, 0)} = \mb{q}(0)$.
In our case, it is simply ${q}_{(0, 0)} = 0$.

We test three representative frequency combinations that span different regimes.
The standard case ($\omega_0 = 1$, $\omega_f = \sqrt{2}$) uses frequencies of comparable magnitude whose ratio is a classical irrational number, producing a quasi-periodic trajectory that densely fills the torus.
The large frequency ratio case ($\omega_0 = 2\pi/100$, $\omega_f = 1$) tests the method's ability to handle widely separated time scales ($\omega_f/\omega_0 \approx 15.9$), where one frequency completes many cycles while the other progresses slowly.
The beat frequency case ($\omega_0 = 1$, $\omega_f = 0.97 + 0.03\sqrt{2}$) has nearly commensurate frequencies ($\omega_f/\omega_0 \approx 1.012$), producing slow amplitude modulation; this regime is particularly challenging for classical harmonic balance methods.

The resulting time-domain comparisons for these three cases are provided in \cref{fig:pressure}, \cref{fig:vorticity}, and \cref{fig:qcrit}; each panel juxtaposes the torus reconstruction with the direct integration. As can be seen, the TTSM and the time-accurate solutions -- obtained using the classical fourth-order Runge--Kutta scheme -- are in perfect agreement even though only 3 discrete points are used in the $\theta_1$ and $\theta_2$ directions, i.e., $n_{\text{TS},1} = n_{\text{TS},2} = 3$. 
\begin{figure}[htbp]
\centering
\begin{subfigure}[t]{\textwidth}
\centering
\includegraphics[width=\linewidth]{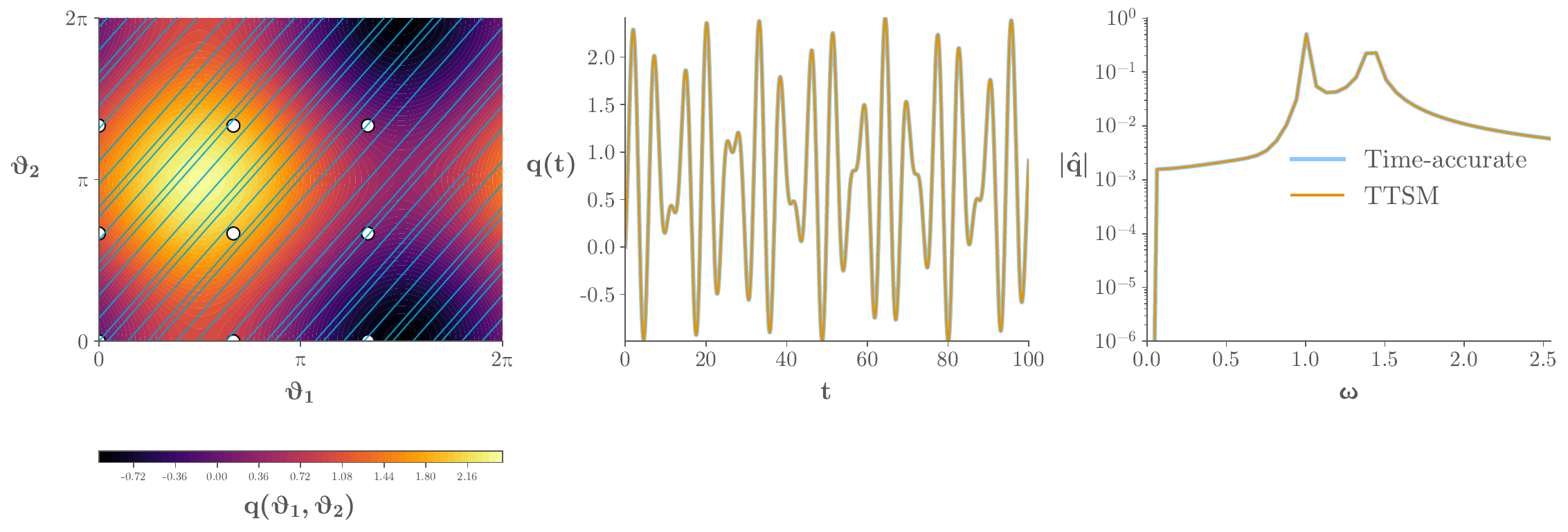}
\caption{$\omega_0=1$ and $\omega_f=\sqrt{2}$.}
\label{fig:pressure}
\end{subfigure}
\vskip1em
\begin{subfigure}[t]{\textwidth}
\centering
\includegraphics[width=\linewidth]{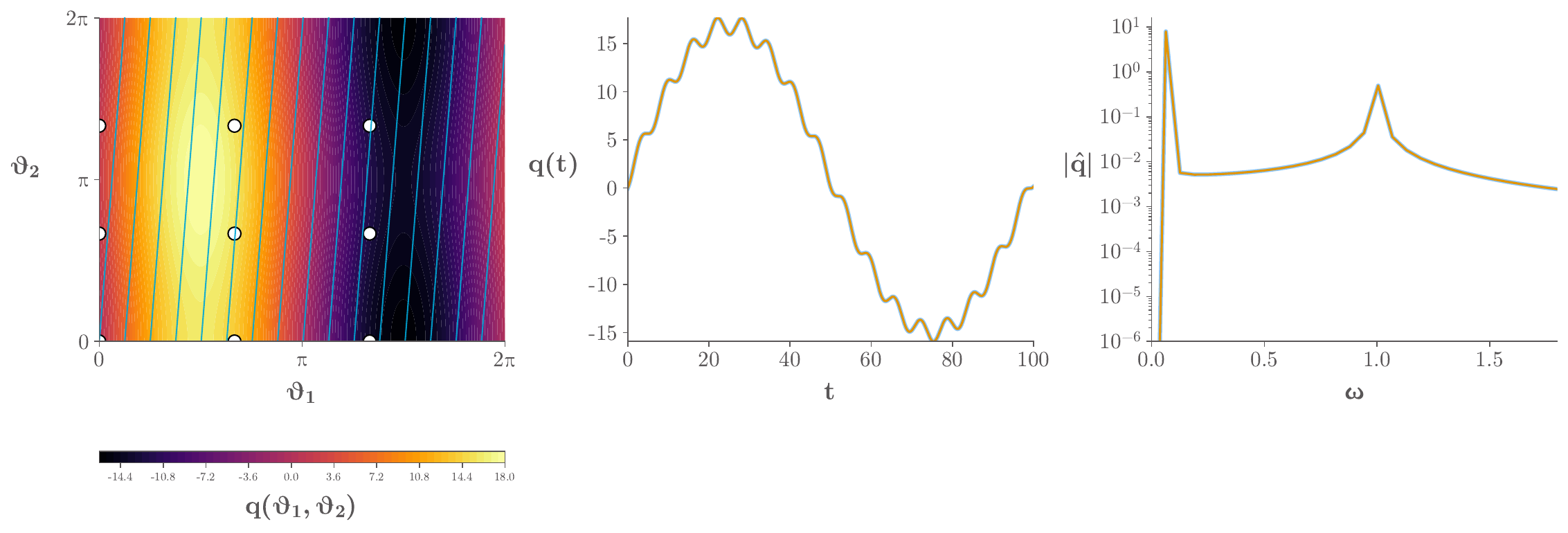}
\caption{$\omega_0=1$ and $\omega_f={2\pi}/{100}$}
\label{fig:vorticity}
\end{subfigure}
\vskip1em
\begin{subfigure}[t]{\textwidth}
\centering
\includegraphics[width=\linewidth]{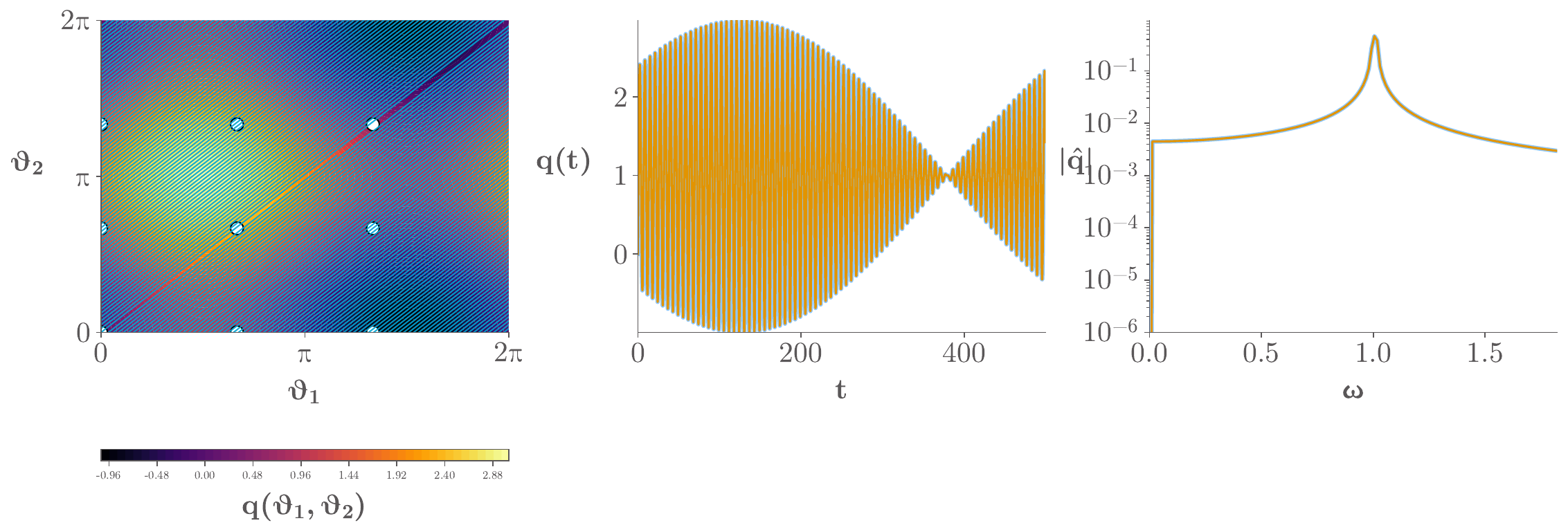}
\caption{$\omega_0=1$ and $\omega_f=0.97 + 0.03 \sqrt{2}$}
\label{fig:qcrit}
\end{subfigure}
\caption{Forced oscillation with two incommensurate frequencies (\cref{eq:examp_1}).
Each row compares the lifted torus representation with the corresponding time history; the markers indicate the initial (circle) and final (diamond) states.}
\label{fig:forced_dual_frequency}
\end{figure}


The beat frequency case in \cref{fig:qcrit} highlights a key advantage of TTSM over the SF-HB method \citep{Li2021}.
In SF-HB, a base frequency $\omega_{\text{base}}$ must satisfy $\omega_0 = n_1 \omega_{\text{base}}$ and $\omega_f \approx n_2 \omega_{\text{base}}$ for integers $n_1, n_2$.
For the beat case with $\omega_0 = 1$ and $\omega_f \approx 1.0124$, this requires $\omega_{\text{base}} \approx 0.0125$ and $n_{\text{h}} = \max(n_1, n_2) = 81$ harmonics, yielding $2 n_{\text{h}} + 1 = 163$ degrees of freedom (DOFs).
The TTSM discretizes the torus directly in $(\theta_1, \theta_2)$ space, where the required resolution depends only on solution smoothness---not the frequency ratio---so a $3 \times 3$ grid (9 DOFs) suffices for machine precision.
The primary advantage of TTSM is the $18\times$ reduction in DOFs, which grows as $\omega_f/\omega_0 \to 1$.
\Cref{fig:sfhb_comparison} illustrates this: SF-HB requires 163 collocation points over a pseudo-period $T \approx 503$, whereas TTSM uses only 9 points on the torus with equivalent accuracy.

\begin{figure}[htbp]
\centering
\includegraphics[width=\linewidth]{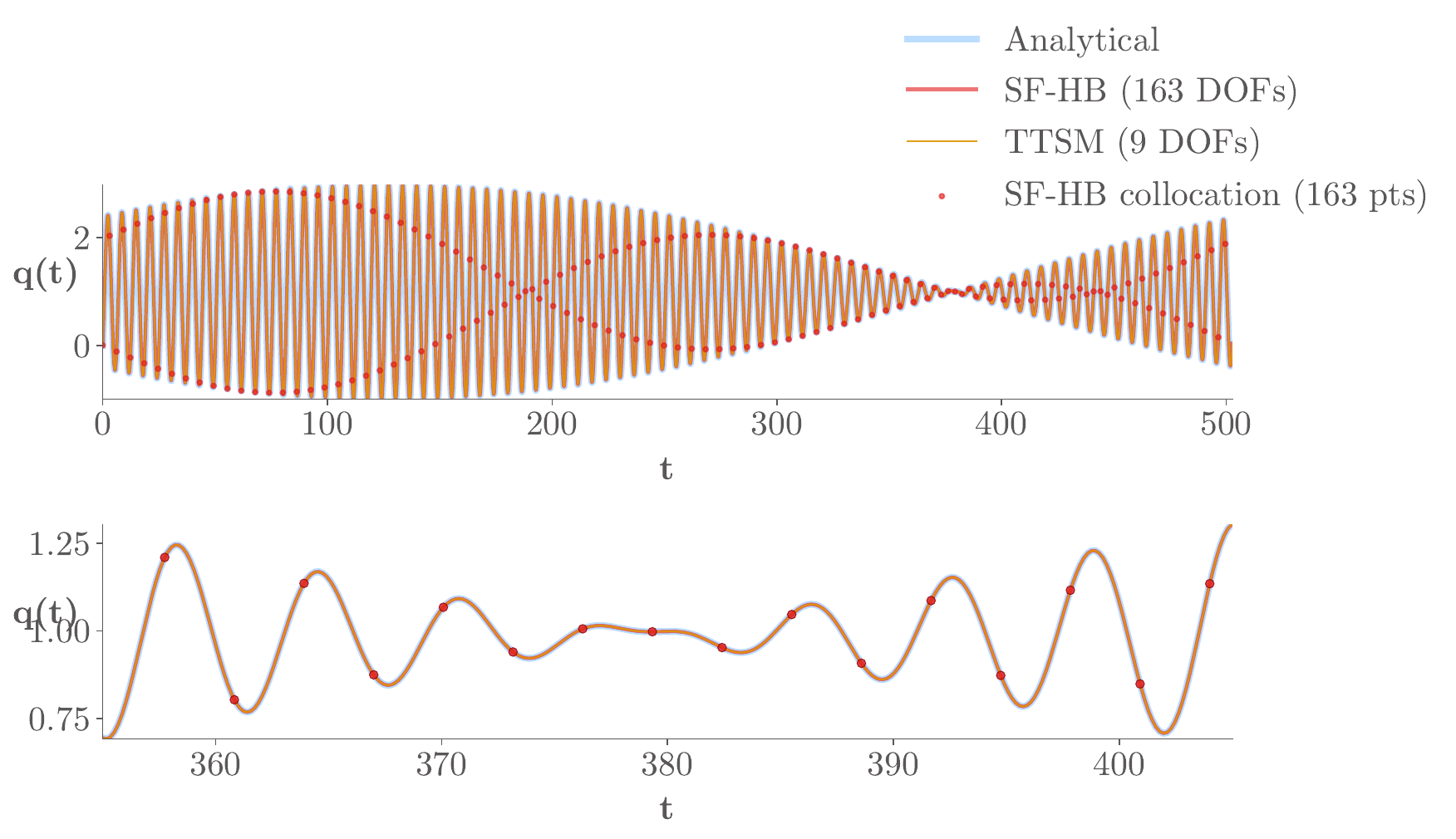}
\caption{Comparison of SF-HB and TTSM for the linear forced oscillator with beat frequencies ($\omega_0 = 1$, $\omega_f \approx 1.0124$).
Top: quasi-periodic response over one SF-HB pseudo-period ($T \approx 503$) with all 163 collocation points marked.
Bottom: zoomed view near the beat node showing the TTSM reconstruction from only 9 torus collocation points.
Both methods achieve machine-precision accuracy, but TTSM requires $18\times$ fewer degrees of freedom.}
\label{fig:sfhb_comparison}
\end{figure}

\subsection{Duffing oscillator}
\label{sec:duffing_convergence}

Having established the method on a linear problem with a closed-form reference, we now turn to a nonlinear benchmark: the two-tone Duffing oscillator.
The cubic stiffness couples the two forcing tones, generating combination frequencies (e.g., $2\omega_1 - \omega_2$) that are absent in the linear case, making this a stringent test of the TTSM framework.

The nondimensional Duffing oscillator governed by
\begin{equation}
\ddot{q} + \delta \dot{q} + \beta q + \alpha q^{3}
= f_{1}\cos(\omega_{1} t) + f_{2}\cos(\omega_{2} t),
\label{eq:duffing}
\end{equation}
with the parameter set
\begin{equation}
\delta = 0.1,~~
\beta = 1.0,~~
\alpha = 3.0,~~
f_1 = 0.05,~~
f_2 = 0.04,~~
(\omega_{1}, \omega_{2}) = \bigl(1,\sqrt{2}\,\bigr).
\end{equation}
We obtain this moderately nonlinear regime by homotopy from a weaker operating point (\,$\alpha=1$, $f_1=0.02$, $f_2=0.015$\,), which provides a robust initial guess for the TTSM solution.
Introducing the lifted state $\mb{q} = [q,\,\dot{q}]^\intercal$, \cref{eq:duffing} can be written as the invariance equation
\begin{equation}
\omega_{1}\,\partial_{\theta_{1}}\mb{q}(\theta_{1},\theta_{2})
+ \omega_{2}\,\partial_{\theta_{2}}\mb{q}(\theta_{1},\theta_{2})
=
\begin{bmatrix}
\dot{q} \\
-\delta \dot{q} - \beta q - \alpha q^{3}
+ f_1\cos\theta_{1} + f_2\cos\theta_{2}
\end{bmatrix},
\label{eq:duffing_torus_rhs}
\end{equation}
posed on the torus $(\theta_{1},\theta_{2}) \in [0,2\pi)^2$ with $\theta_{j} = \omega_{j} t$.
This system is then solved using the generic torus collocation solver described earlier.
A tensor-product grid with $n_{\text{TS},1} = n_{\text{TS},2} = 3$ suffices to resolve this moderately nonlinear response.
The nonlinear algebraic system was converged to a residual tolerance of $10^{-10}$, producing the torus state $\mb{q}(\theta_{1},\theta_{2})$ and its Fourier coefficients.

To assess the TTSM solution accuracy consistent with \cref{thm:convergence}, we compare the $3\times3$ solution to a fine-grid reference ($31\times31$) interpolated to the coarse collocation points.
The resulting solution error is $\|\hat{\mb{q}}-\hat{\mb{v}}\|_\infty = 1.6\times10^{-2}$.
For visualization, a companion time-accurate integration of \cref{eq:duffing} was performed over $t \in [0,220]$ with $1.5\times10^{4}$ samples.  The first $55$ units were discarded to remove transients; because the system is non-autonomous, the torus solution is evaluated directly at phases $\theta_j = \omega_j t$ without any alignment.
The right panel of \cref{fig:duffing_beat} shows the frequency spectra computed from the post-transient window ($t \in [55,220]$): peaks appear at the two forcing frequencies $\omega_1$ and $\omega_2$, as well as at the combination tones generated by the cubic nonlinearity.
The TTSM spectrum coincides with the time-accurate result, confirming that the coarse $3\times3$ grid captures the essential nonlinear frequency content.
The resulting three-panel overview, shown in \cref{fig:duffing_beat}, highlights the torus field, the transient comparison, and the matching frequency spectra for a representative parameter set, while the grid-refinement sweep is summarized in \cref{fig:duffing_convergence}.
\begin{figure}[htbp]
\centering
\includegraphics[width=1.0\linewidth]{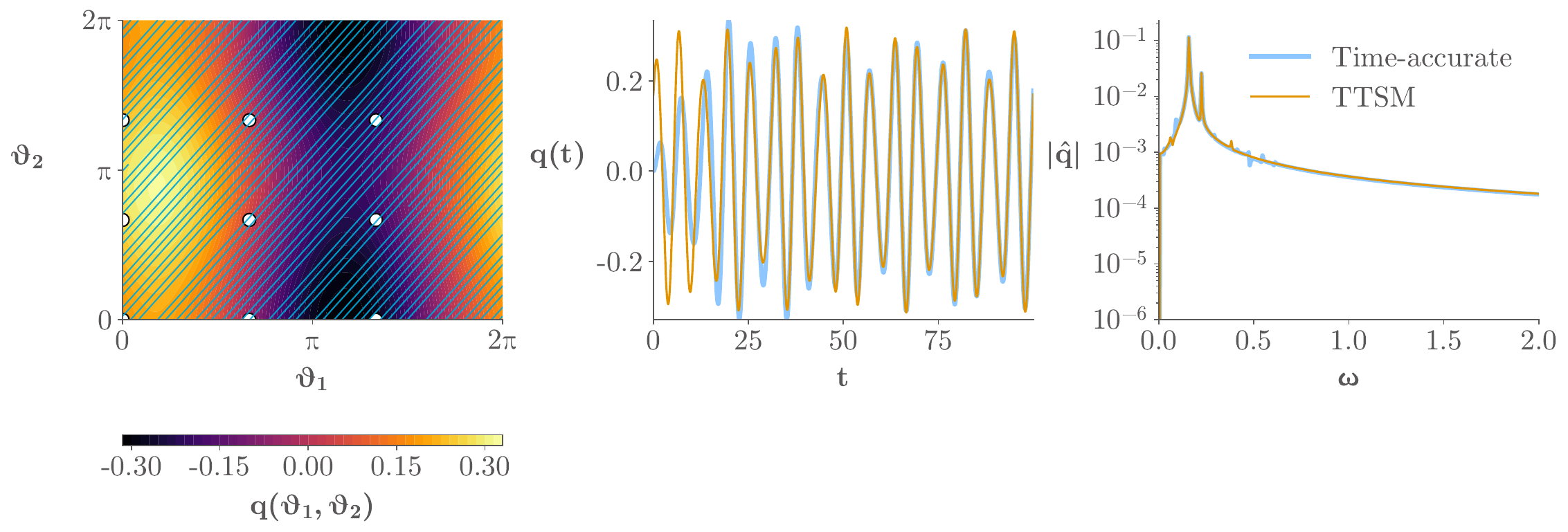}
\caption{Torus collocation of the two-tone Duffing oscillator.
Left: torus field $q(\theta_{1},\theta_{2})$ with the trajectory and collocation nodes.
Middle: comparison between the torus solution and the time-accurate integration during the transient phase ($t \in [0,100]$); the mismatch arises because the time-domain trajectory starts from zero initial conditions and must decay toward the attracting torus.
Right: frequency spectra of the two responses showing coincident peaks.}
\label{fig:duffing_beat}
\end{figure}

The middle panel of \cref{fig:duffing_beat} illustrates a key advantage of the torus formulation: the TTSM solution captures the fully developed quasi-periodic attractor directly, bypassing the transient decay that dominates the early portion of time-accurate integrations.
For forced-dissipative systems such as the Duffing oscillator, the torus solution exists as an attracting invariant manifold, and time-domain trajectories converge to it regardless of initial conditions.
The TTSM exploits this structure by solving for the invariant torus itself, eliminating the need to integrate through potentially long transients.

\begin{figure}[htbp]
\centering
\includegraphics[width=\linewidth]{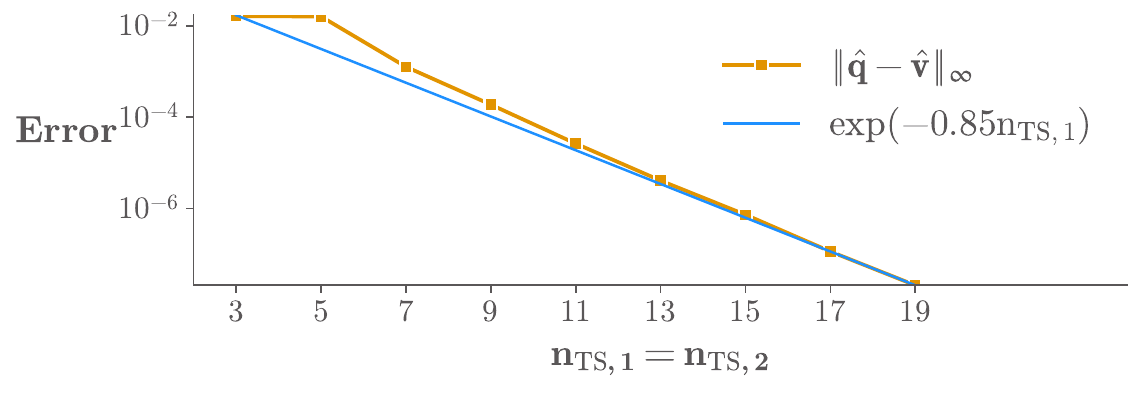}
\caption{Convergence of the Duffing oscillator solution as the tensor grid is refined ($n_{\text{TS},1}=n_{\text{TS},2}\in\{3,5,7,9,11,13,15,17,19\}$).
The solution error $\|\hat{\mb{q}} - \hat{\mb{v}}\|_\infty$ is measured relative to a fine-grid reference; the solid line shows an exponential fit to the two finest grids.}
\label{fig:duffing_convergence}
\end{figure}

\Cref{fig:duffing_convergence} confirms the spectral convergence predicted by \cref{thm:convergence}.
The solution error decays exponentially with grid size, consistent with the analyticity of the Duffing torus solution.
The fitted decay rate of approximately $0.85$ per grid point reflects the smoothness of the underlying attractor.
Even with only $3 \times 3 = 9$ collocation points, the error remains below $2\%$, demonstrating that the TTSM achieves accurate quasi-periodic solutions with modest computational effort.

\subsection{Nonlinear Klein--Gordon equation}

The model problem considered here is the nonlinear Klein--Gordon equation in one spatial dimension,
\begin{equation}
q_{tt} - q_{xx} + q + \varepsilon q^3 + \gamma q_t
= g \sin(x)\,\bigl(\cos(\omega_1 t) + \cos(\omega_2 t)\bigr),
\label{eq:kg}
\end{equation}
where $q(x,t)$ denotes the field variable.
The parameter $\gamma$ introduces damping, $\varepsilon$ controls the strength of nonlinearity, and $g$ sets the amplitude of an external spatially varying forcing.
The right-hand side includes driving terms at two incommensurate frequencies $\omega_1$ and $\omega_2$, which makes the long-time response quasi-periodic and naturally suited for a torus representation.

Introducing the velocity $v = q_t$, \cref{eq:kg} can be rewritten as a first-order system in time:
\begin{equation}
\begin{bmatrix} q_t \\ v_t \end{bmatrix}
=
\begin{bmatrix}
v \\
q_{xx} - q - \varepsilon q^3 - \gamma v + g \sin(x)\,\bigl(\cos(\omega_1 t) + \cos(\omega_2 t)\bigr)
\end{bmatrix}.
\label{eq:kg_first_order}
\end{equation}
For these computations, the spatial direction $x$ is discretized using finite differences, converting the PDE into a system of coupled ODEs.
A direct time-integration of this system produces a high fidelity time-accurate reference solution.
In parallel, the problem is reformulated in relation to the torus defined by the two forcing frequencies by introducing the angular coordinates $\theta_1 = \omega_1 t$ and $\theta_2 = \omega_2 t$.
The lifted state $\mb{q} = [q,\,v]^\intercal$ then satisfies the invariance equation
\begin{equation}
\omega_{1}\,\partial_{\theta_{1}}\mb{q}(\theta_{1},\theta_{2})
+ \omega_{2}\,\partial_{\theta_{2}}\mb{q}(\theta_{1},\theta_{2})
=
\begin{bmatrix}
v \\
q_{xx} - q - \varepsilon q^3 - \gamma v + g \sin(x)\,\bigl(\cos\theta_1 + \cos\theta_2\bigr)
\end{bmatrix},
\label{eq:kg_torus_rhs}
\end{equation}
posed on the torus $(\theta_{1},\theta_{2}) \in [0,2\pi)^2$ with $\theta_{j} = \omega_{j} t$.
In this time-spectral approach, the governing equations are enforced at collocation points on the two-dimensional torus, and the unknown field is recovered by solving a nonlinear algebraic system for its Fourier coefficients.
We use the parameter set $\omega_1 = 1.0$, $\omega_2 = \sqrt{2}$, $\gamma = 0.2$, $\varepsilon = 0.5$, and $g = 1.0$.
For visualization, a companion time-accurate integration was performed over $t \in [0,200]$ starting from zero initial conditions $(q,\dot{q})=(0,0)$.
The first half of the time record was discarded to allow transients to decay; because the system is non-autonomous, the torus solution is evaluated directly at phases $\theta_j = \omega_j t$ without any alignment.

\Cref{fig:kg_global} first highlights the space--time contour of the reference solution, whereas \cref{fig:kg_results} presents a three-row summary: each row corresponds to one probe location ($x=1.05$, $x=1.75$, $x=2.44$) and contains the torus field $q(\theta_1,\theta_2)$ with the wrapped trajectory and collocation nodes, the long-time histories of $q$, and the frequency spectrum of $q$.
In addition to these local measures, we report the solution error $\|\hat{\mb{q}} - \hat{\mb{v}}\|_\infty$ relative to a fine-grid reference, following the same methodology as the Duffing convergence study.  These results clearly show the efficacy of the novel TTSM technique for a nonlinear and {\em aperiodic} dynamic problem. 
\begin{figure}[htbp]
\centering
\includegraphics[width=0.85\linewidth]{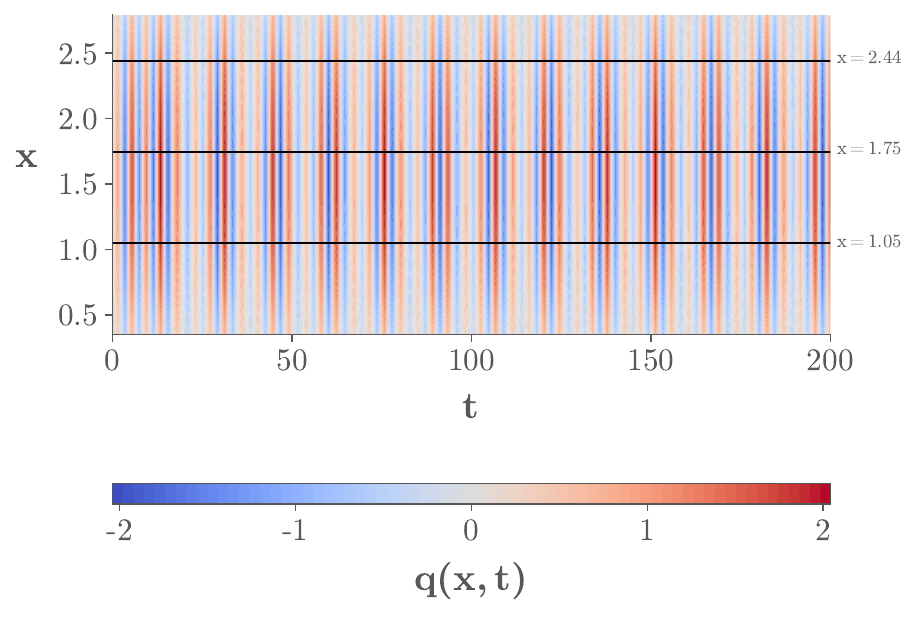}
\caption{Space--time contour of the time-accurate Klein--Gordon solution $q(x,t)$ used for validating the torus reconstruction.}
\label{fig:kg_global}
\end{figure}

\begin{figure}[htbp]
\centering
\includegraphics[width=1.0\linewidth]{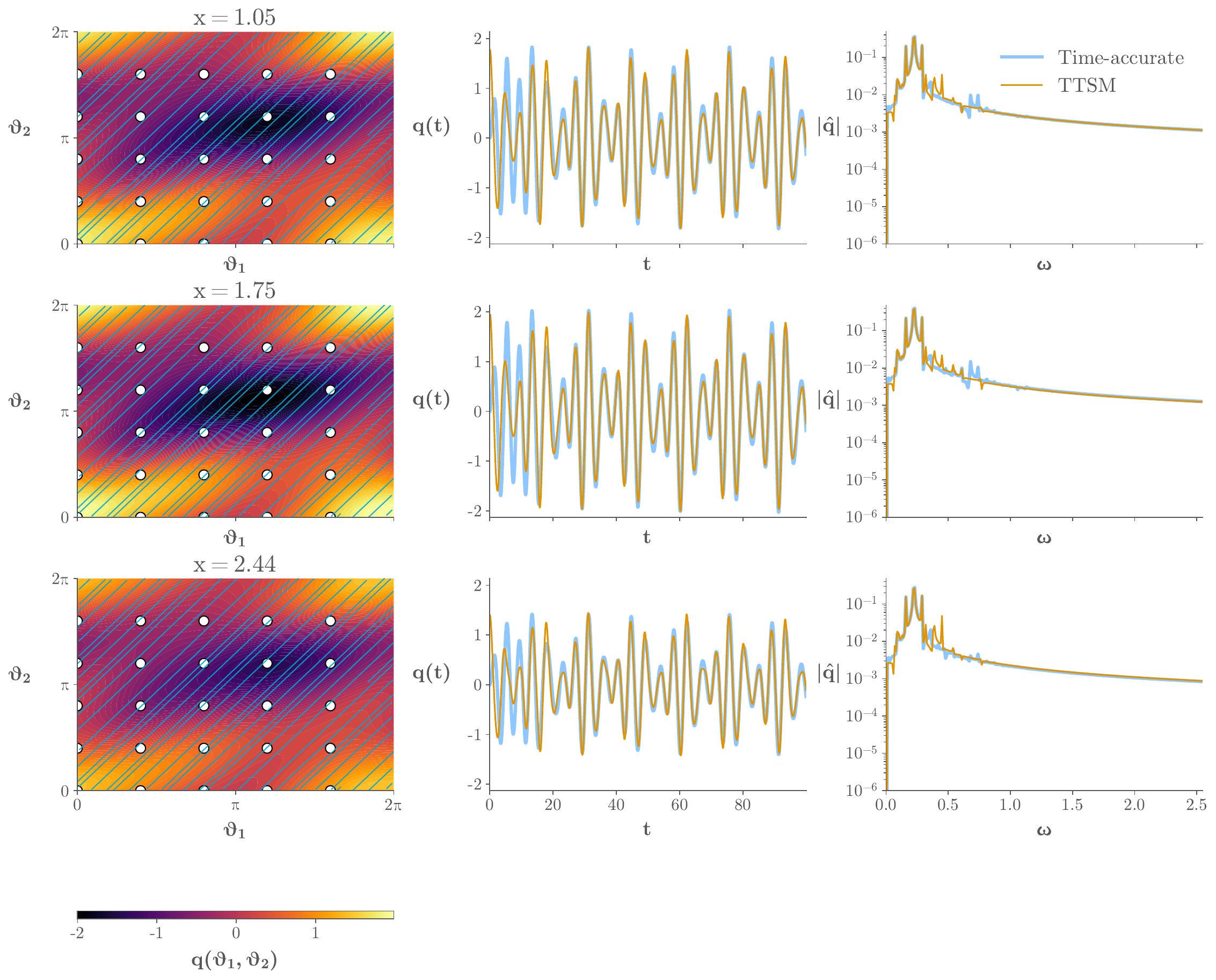}
\caption{Klein--Gordon torus solution for $n_{\text{TS},1} = n_{\text{TS},2} = 5$.
Each row corresponds to one spatial probe ($x=1.05$, $x=1.75$, $x=2.44$).
Left column: torus field $q(\theta_1,\theta_2)$ for that station together with the trajectory and collocation nodes.
Middle column: comparison between the torus reconstruction and the time-accurate integration during the transient phase ($t \in [0,100]$).
Right column: frequency spectrum of $q$ at the same station, highlighting coincident tones.}
\label{fig:kg_results}
\end{figure}

\Cref{fig:kg_results} demonstrates that a modest $5\times5$ tensor grid suffices to capture the quasi-periodic response at each spatial station.
The left column shows the torus field $q(\theta_1,\theta_2)$ exhibiting smooth variation consistent with the analytic regularity assumed in \cref{thm:convergence}.
The middle column compares the torus reconstruction (evaluated at raw phases $\theta_j = \omega_j t$ without alignment) with time-accurate integration during the transient phase ($t \in [0,100]$); the mismatch arises because the time-domain trajectory starts from zero initial conditions and must decay toward the attracting torus, while the TTSM solution represents the fully developed quasi-periodic attractor from the outset.
The right column displays the frequency spectrum, where the dominant peaks at $\omega_1=1$ and $\omega_2=\sqrt{2}$ are clearly resolved, along with combination tones arising from the cubic nonlinearity.
The spatial variation across rows reflects how the forcing profile $\sin(x)$ modulates the local amplitude, yet the underlying two-frequency structure remains consistent throughout the domain.

\Cref{fig:kg_convergence} confirms the spectral convergence predicted by \cref{thm:convergence}.
The solution error decays exponentially with grid size, consistent with the analyticity of the Klein--Gordon torus solution.
The fitted decay rate of approximately $0.55$ per grid point reflects the smoothness of the underlying quasi-periodic attractor.
Notably, the Klein--Gordon system involves $2 n_x = 16$ state variables per torus grid point (displacement and velocity at each of the $n_x = 8$ spatial nodes), yet the TTSM resolves all coupled modes with a modest $5 \times 5$ tensor grid.
This demonstrates the method's ability to handle spatially extended systems where the quasi-periodic structure arises from temporal forcing rather than spatial periodicity.
\begin{figure}[htbp]
\centering
\includegraphics[width=\linewidth]{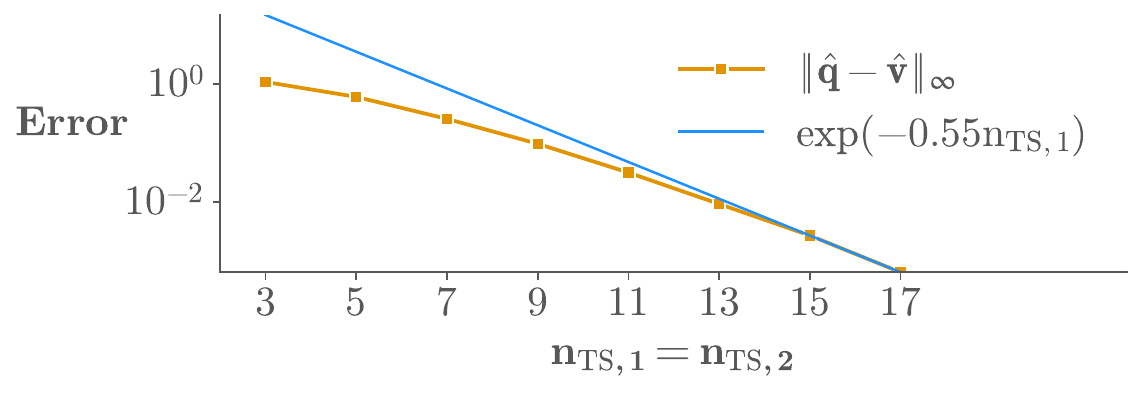}
\caption{Convergence of the nonlinear Klein--Gordon solution as the tensor grid is refined ($n_{\text{TS},1}=n_{\text{TS},2}\in\{3,5,7,9,11,13,15,17,19\}$).
The solution error $\|\hat{\mb{q}} - \hat{\mb{v}}\|_\infty$ is measured relative to a fine-grid reference; the solid line shows an exponential fit to the two finest grids.}
\label{fig:kg_convergence}
\end{figure}

\section{Conclusion}
\label{sec:conclusion}

This paper introduced the torus time-spectral method (TTSM) for quasi-periodic dynamical systems driven by two incommensurate frequencies.
By reformulating the governing equations on a two-dimensional torus $\mathbb{T}^2$ and applying double-Fourier collocation, we transform the infinite-horizon quasi-periodic problem into a finite nonlinear algebraic system that can be solved with standard Newton--Krylov techniques.

The key contributions of this work are twofold.
First, we developed a time-domain torus collocation formulation that enables direct reuse of existing residual routines from steady-state or single-frequency time-spectral solvers, and showed that the TTSM bypasses the transient decay phase inherent to time-accurate integration by directly computing the attracting invariant torus.
Second, we established a rigorous convergence analysis (\cref{thm:convergence}) guaranteeing spectral accuracy for analytic solutions, and verified through numerical experiments that modest tensor grids---as few as $3\times3$ for the Duffing oscillator and $5\times5$ for the spatially discretized Klein--Gordon equation---suffice to achieve engineering accuracy, with solution errors decaying exponentially at rates of approximately $0.55$--$0.85$ per grid point.

The method does have limitations.
The frequencies $\omega_1$ and $\omega_2$ must be known \emph{a priori}, which restricts applicability to problems where the forcing frequencies are prescribed or can be identified from preliminary analysis; autonomous LCOs, whose frequency emerges from the dynamics itself rather than from external forcing, would require an augmented formulation that solves for the unknown period alongside the torus solution.
Additionally, the non-resonance condition underlying \cref{thm:convergence} may be violated for certain frequency ratios or parameter regimes, leading to ill-conditioned Jacobians.
For problems with neutral directions (e.g., undamped linear oscillators), a nodal anchor or similar regularization is required to remove the resulting gauge freedom.




\bibliographystyle{elsarticle-num}
\bibliography{bib/references}  

\end{document}